\documentclass[11pt]{article}

\usepackage[a4paper,margin=1in]{geometry}
\usepackage{amsmath,amssymb,amsfonts,amsthm,mathtools}
\usepackage{mathrsfs}
\usepackage{enumerate}
\usepackage{enumitem}
\usepackage{graphicx}
\usepackage{hyperref}
\usepackage{xcolor}
\usepackage{tikz-cd}

\hypersetup{
  colorlinks=true,
  linkcolor=blue,
  citecolor=blue,
  urlcolor=blue
}

\numberwithin{equation}{section}

\newtheorem{theorem}{Theorem}[section]
\newtheorem{proposition}[theorem]{Proposition}
\newtheorem{lemma}[theorem]{Lemma}

\theoremstyle{definition}
\newtheorem{definition}[theorem]{Definition}

\newtheorem{remark}[theorem]{Remark}

\title{A HOMFLYPT-type invariant for pseudo links via a resolution in Hecke algebras}
\author{
Ioannis Diamantis\\
\small Department of Data Analytics and Digitalisation, School of Business and Economics\\
\small Maastricht University, Maastricht, The Netherlands\\
\small \texttt{i.diamantis@maastrichtuniversity.nl}
}

\date{}

\begin{document}

\maketitle

\begingroup
\renewcommand\thefootnote{}
\endgroup

\begin{abstract}
Pseudo links generalize classical links by allowing crossings with missing over/under information, called pre-crossings. While the pseudo braid framework provides an algebraic description of pseudo links via a Markov-type theorem, the construction of polynomial invariants using Hecke algebra techniques is obstructed by the presence of the pseudo Reidemeister 1 move. In this paper, we construct a HOMFLYPT-type invariant for oriented pseudo links via the pseudo Hecke algebra of type \(A\). The construction is based on a resolution homomorphism that maps each pseudo generator to a linear combination of a braid generator and its inverse, interpreting pre-crossings as algebraic superpositions of classical crossings. Composing this map with the Ocneanu trace and applying a suitable normalization yields an invariant satisfying a natural pseudo skein relation. We further show that the invariant admits a state-sum formulation as a weighted sum of classical HOMFLYPT-type invariants over all classical resolutions of the pseudo crossings, as well as a skein-theoretic characterization in terms of its values on classical links and the pseudo skein relation.
\end{abstract}


\section{Introduction}\label{intro}

Pseudo links arise as a natural generalization of classical links by allowing crossings with missing over/under information. These \emph{pre-crossings} encode situations in which the local structure of a link diagram is only partially known. The notion was first introduced by Hanaki \cite{Hanaki2010} and further developed in \cite{Henrich2013}, where pseudo knots are studied as equivalence classes of diagrams under generalized Reidemeister moves. Such structures have found applications, for instance, in the study of DNA configurations, where experimental limitations may prevent the distinction between positive and negative crossings.

The theory of pseudo knots and links has been developed through both diagrammatic and braid-theoretic approaches. In particular, the introduction of the pseudo braid monoid \cite{BardakovJablanWang2016} and the corresponding Markov theorem \cite{BardakovJablanWang2016, D} provide a complete algebraic framework for the study of pseudo links via braids, extending classical results of Alexander \cite{Alexander1923} and Markov \cite{Markov1936}. This framework parallels the classical correspondence between links and braids \cite{Birman1993,Lambropoulou2007}, while incorporating additional moves reflecting the presence of pre-crossings.

A fundamental problem in knot theory is the construction of link invariants. For classical links, Hecke algebras and Markov traces provide a powerful method for producing polynomial invariants, most notably the HOMFLYPT polynomial, following the seminal work of Jones \cite{Jones1987}, where a trace on the Hecke algebra, originally discovered by Ocneanu, is used to construct the HOMFLYPT polynomial. This approach has been extended to singular links via the singular braid monoid and the singular Hecke algebra \cite{Baez1992,ParisRabenda2008}, where additional generators correspond to singular crossings and lead to invariants satisfying desingularization relations. Related algebraic structures have also been developed in other generalized settings; for instance, the bonded braid monoid introduced in \cite{DKL} exhibits a rich algebraic structure closely connected to both singular and pseudo braid monoids.

Despite the close relationship between pseudo links and singular links, the Hecke algebra approach does not directly extend to the pseudo setting. Although the pseudo braid monoid is algebraically isomorphic to the singular braid monoid \cite{BardakovJablanWang2016}, the underlying topological theories differ in a crucial way: pseudo links admit the pseudo Reidemeister 1 move, which has no analogue in singular knot theory. This discrepancy introduces an additional pseudo-stabilization move in the Markov theorem \cite{D, DiamantisPseudoST}, and prevents standard trace constructions from yielding invariants of pseudo links without modification.

In this paper, we construct a HOMFLYPT-type invariant for pseudo links via a Hecke algebra framework. The key idea is to interpret each pre-crossing as a linear combination of a positive and a negative crossing. Algebraically, this is achieved by defining a homomorphism from the pseudo Hecke algebra to the classical Hecke algebra, sending each pseudo generator to a linear combination of the corresponding braid generator and its inverse. Composing this map with the Ocneanu trace yields a trace-like functional on pseudo braids. Conceptually, this construction treats pre-crossings as algebraic superpositions of classical crossings, in a manner analogous to tangle insertion techniques \cite{HenrichKauffman2017}, allowing the pseudo structure to be incorporated directly into the Hecke algebra setting.

A suitable normalization is then introduced to account for the pseudo-stabilization move, leading to an invariant \(\mathcal P\) of pseudo links that extends the classical HOMFLYPT polynomial. At the algebraic level, pre-crossings are resolved according to the relation
\[
L_p = X L_+ + Y L_-,
\]
and this leads, after normalization, to a pseudo skein relation for the invariant \(\mathcal P\). This relation contrasts with the desingularization relations appearing in the singular setting \cite{ParisRabenda2008}, and highlights the distinct topological nature of pseudo links.

In addition to its algebraic construction, the invariant admits a state-sum formulation as a weighted sum over classical resolutions of pseudo crossings, in the spirit of state-sum models such as the Kauffman bracket \cite{Kauffman1988}, as well as a skein-theoretic characterization. These equivalent descriptions provide complementary algebraic and diagrammatic perspectives on the invariant.

\medskip 

The paper is organized as follows. In Section~\ref{preliminaries} we recall the necessary background on pseudo links, pseudo braids, and Hecke algebras. In Section~\ref{pseudoheck} we introduce the pseudo Hecke algebra and define the resolution homomorphism. Section~\ref{trace} is devoted to the construction of the induced trace and the proof of the main theorem. In Section~\ref{skein} we derive the skein relations satisfied by the invariant, including the pseudo skein relation. In Section~\ref{statesum} we develop a state-sum interpretation and a skein-theoretic characterization of the invariant. Finally, in Section~\ref{concl} we discuss further directions, including extensions to 3-manifolds and related algebraic structures.


\section{Preliminaries}\label{preliminaries}

In this section we recall the necessary background on pseudo links, pseudo braids and singular Hecke algebras from \cite{Henrich2013,BardakovJablanWang2016,ParisRabenda2008}.

\subsection{Pseudo links and pseudo knots}

Pseudo diagrams were introduced by Hanaki \cite{Hanaki2010} as
generalizations of classical link diagrams in which some crossings
lack over/under information. These \emph{pre-crossings} may be interpreted as representing uncertainty between the two possible classical crossings (for an illustration see Figure~\ref{pseudoknot}).

\begin{figure}[ht]
\centering
\includegraphics[scale=1]{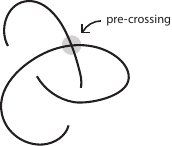}
\caption{An example of a pseudo knot diagram containing one pre-crossing.}
\label{pseudoknot}
\end{figure}

\begin{definition}
A \emph{pseudo link diagram} is a regular link diagram in which some crossings are replaced by \emph{pre-crossings}, that is, transverse double points with no specified over/under information.

Two pseudo link diagrams are said to be \emph{equivalent} if they are related by a finite sequence of generalized Reidemeister moves, consisting of the classical Reidemeister moves together with the pseudo Reidemeister moves \cite{Henrich2013}, illustrated in Figure~\ref{Reidemeister}.

A \emph{pseudo link} is an equivalence class of pseudo link diagrams under these moves.
\end{definition}

\begin{figure}[ht]
\centering
\includegraphics[scale=0.75]{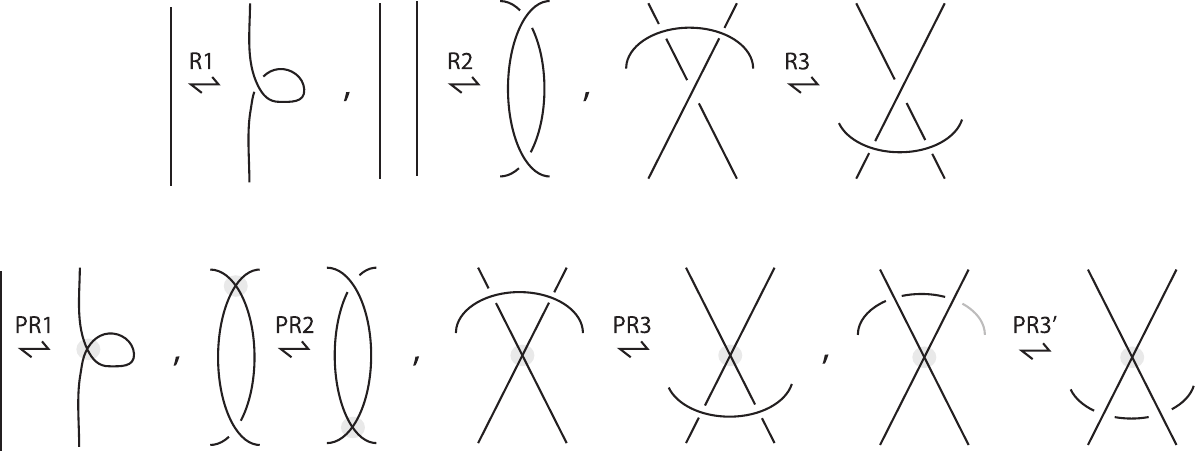}
\caption{Reidemeister Moves for Pseudo Knots.}
\label{Reidemeister}
\end{figure}

\begin{remark}
Pseudo knots are closely related to singular knots, that is, knots with finitely many transverse double points regarded as rigid vertices. Indeed, there is a natural correspondence between singular diagrams and pseudo diagrams obtained by replacing each singular crossing with a pre-crossing \cite{BardakovJablanWang2016, DiamantisPseudoST}. However, this correspondence is not an equivalence of theories. In particular, the pseudo Reidemeister 1 move, which allows the creation or removal of a loop with a pre-crossing, has no analogue in the
singular setting. This distinction plays a crucial role in the construction of invariants.
\end{remark}


\subsection{Pseudo braids and the Markov theorem}

The braid-theoretic approach to pseudo links is based on the
\emph{pseudo braid monoid}, introduced in \cite{BardakovJablanWang2016}. This construction extends the classical braid group by allowing braid diagrams with pre-crossings.

\begin{definition}\label{def:pbrmon}
The \emph{pseudo braid monoid} \(PM_n\) is the monoid generated by
\[
\sigma_1^{\pm1},\ldots,\sigma_{n-1}^{\pm1},
\qquad
p_1,\ldots,p_{n-1},
\]
where the generators \(\sigma_i^{\pm1}\) are the classical braid
generators and the generators \(p_i\) represent pre-crossings between the \(i\)-th and \((i+1)\)-st strands (for an illustration see Figure~\ref{gen}).

The defining relations are the classical braid relations among the
\(\sigma_i\)'s, together with the pseudo braid relations
\[
p_i p_j = p_j p_i, \qquad |i-j|\geq 2,
\]
\[
p_i \sigma_j^{\pm1} = \sigma_j^{\pm1} p_i, \qquad |i-j|\geq 2,
\]
\[
p_i \sigma_i^{\pm1} = \sigma_i^{\pm1} p_i,
\]
\[
\sigma_i \sigma_{i+1}p_i = p_{i+1}\sigma_i\sigma_{i+1},
\]
and
\[
\sigma_{i+1}\sigma_i p_{i+1}=p_i\sigma_{i+1}\sigma_i.
\]
\end{definition}

\begin{figure}[ht]
\centering
\includegraphics[scale=0.8]{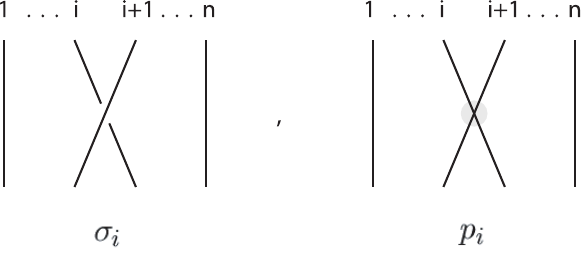}
\caption{The generators of \(PM_n\): classical crossings \(\sigma_i^{\pm1}\) and pre-crossings \(p_i\).}
\label{gen}
\end{figure}

\begin{remark}
The relations above are formally identical to the defining relations of the singular braid monoid after replacing each singular generator \(\tau_i\) by the pseudo generator \(p_i\). Consequently, there is an isomorphism
\[
PM_n \cong SM_n,
\qquad
\sigma_i^{\pm1}\mapsto \sigma_i^{\pm1},
\qquad
p_i\mapsto \tau_i,
\]
where \(SM_n\) denotes the singular braid monoid
\cite{BardakovJablanWang2016}. This isomorphism is algebraic; it does not identify the corresponding topological theories, because, as noted before, pseudo links admit the pseudo Reidemeister 1 move.
\end{remark}

As in the classical case, the closure of a pseudo braid is obtained by joining corresponding top and bottom endpoints (see Figure~\ref{closure}). This produces an oriented pseudo link. More precisely, we have:

\begin{definition}
The \emph{closure} of a pseudo braid \(\alpha \in PM_n\), denoted
\(\widehat{\alpha}\), is the pseudo link obtained by joining the
corresponding top and bottom endpoints of the braid diagram.
\end{definition}

\begin{figure}[ht]
\centering
\includegraphics[scale=0.8]{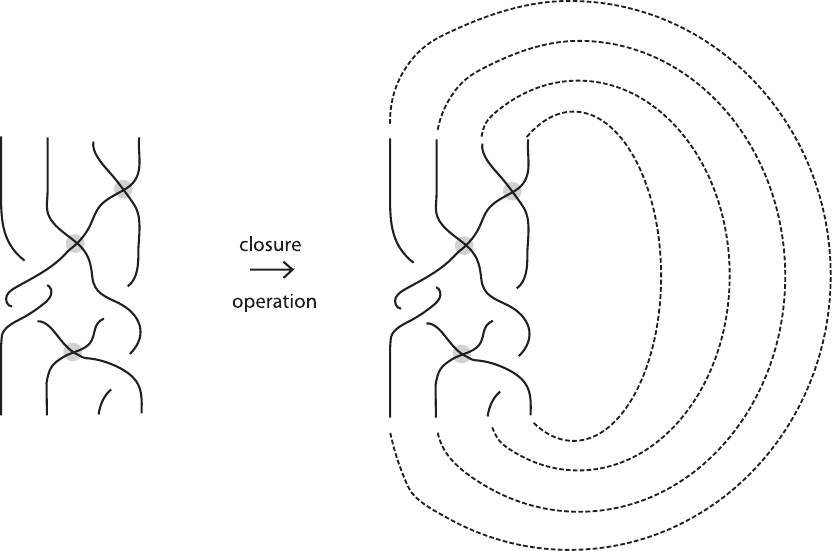}
\caption{A pseudo braid and its closure.}
\label{closure}
\end{figure}

This construction allows one to pass from pseudo braids to pseudo links, providing the bridge between the algebraic braid setting and the topological theory of pseudo links.

\begin{theorem}[Alexander theorem for pseudo links {\cite{BardakovJablanWang2016, D}}]
Every oriented pseudo link is isotopic to the closure of a pseudo braid.
\end{theorem}

This result shows that the study of pseudo links can be reduced to the study of pseudo braids, mirroring the classical braid–link correspondence. The equivalence of closed pseudo braids is described by a Markov-type theorem. We state the form that will be used in this paper.

\begin{theorem}[Markov theorem for pseudo braids {\cite{BardakovJablanWang2016}}]
Two pseudo braids have isotopic closures if and only if they are related by a finite sequence of pseudo braid relations and the following moves (see Figure~\ref{Mar}):
\begin{enumerate}
    \item \emph{conjugation}
    \[
    \alpha \sim \beta^{-1}\alpha\beta,
    \]
    where \(\alpha\in PM_n\) and \(\beta\in B_n\);

    \item \emph{commuting}
    \[
    \alpha\beta \sim \beta\alpha,
    \]
    for \(\alpha,\beta\in PM_n\);

    \item \emph{classical stabilization and destabilization}
    \[
    \alpha \sim \alpha\sigma_n^{\pm1};
    \]

    \item \emph{pseudo-stabilization}
    \[
    \alpha \sim \alpha p_n.
    \]
\end{enumerate}
\end{theorem}

\begin{figure}[ht]
\centering
\includegraphics[scale=0.9]{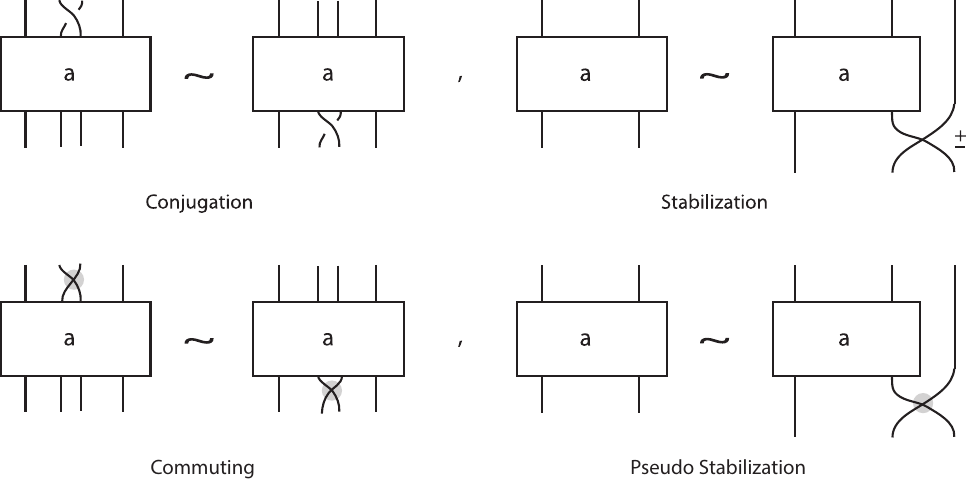}
\caption{Conjugation, Stabilzation, Commuting and Pseudo Stabilization moves.}
\label{Mar}
\end{figure}

\begin{remark}
A sharpened version of the Markov theorem for pseudo links, formulated in terms of $L$-moves, was introduced in \cite{D}. This approach parallels the classical $L$-move framework \cite{Lambropoulou2007} and provides a more geometric description of braid equivalence.
\end{remark}

The pseudo-stabilization move is the braid-theoretic counterpart of the pseudo Reidemeister 1 move. This feature distinguishes the pseudo setting from the singular case and constitutes the main obstruction to extending standard Hecke algebra constructions. In particular, classical Markov traces are not compatible with pseudo-stabilization, and any invariant of pseudo links must incorporate an additional normalization depending on the number of pseudo crossings (see \cite{DiamantisPseudoST}).

This phenomenon will be reflected algebraically in Section~\ref{pseudoheck}, where the structure of \(PM_n\) leads to the definition of a pseudo Hecke algebra and necessitates a modification of the classical trace normalization.


\subsection{Singular Hecke algebras}

The construction of the HOMFLYPT polynomial for classical links is based on the existence of a Markov trace on the Iwahori--Hecke algebras of type \(A\), obtained as a quotient of the braid group algebra. This trace is compatible with the Markov moves and yields a link invariant via a suitable normalization \cite{Jones1987}. Extensions of this approach to singular links have been developed using the singular braid monoid and the corresponding singular Hecke algebra, where invariants satisfy desingularization relations \cite{ParisRabenda2008}. In this setting, the singular braid monoid provides the algebraic framework by extending the classical braid group with generators corresponding to singular crossings.

\begin{definition}
The \emph{singular braid monoid} \(SM_n\) is generated by the classical braid generators \(\sigma_i^{\pm1}\) together with singular generators \(\tau_i\), subject to the classical braid relations and additional relations involving the \(\tau_i\)'s (see \cite{Baez1992,ParisRabenda2008}).
\end{definition}

Passing to the monoid algebra and imposing quadratic Hecke relations on the classical generators leads to the corresponding singular Hecke algebra.

\begin{definition}
The \emph{singular Hecke algebra} \(SH_n(q)\) is defined as the quotient
\[
SH_n(q) = R[SM_n] \Big/ \langle g_i^2 - (q-1)g_i - q \rangle,
\]
where \(g_i\) denotes the image of \(\sigma_i\).
\end{definition}

In contrast to the classical case, the singular generators \(\tau_i\) do not satisfy quadratic relations. Instead, they interact with the generators \(g_i\) through relations inherited from the singular braid monoid. A key feature of singular link invariants is the presence of \emph{desingularization relations}, expressing a singular crossing as a linear combination of classical diagrams. More precisely, let \(L_\times, L_+, L_-\), and \(L_0\) denote link diagrams that are identical outside a small disk, and inside the disk differ by a singular crossing, a positive crossing, a negative crossing, and the oriented smoothing, respectively (see Figure~\ref{skeins}). We have:

\begin{theorem}[Paris--Rabenda {\cite{ParisRabenda2008}}]
There exist Markov traces on the tower of singular Hecke algebras
\[
\mathrm{tr}_n : SH_n(q) \to R, \quad n \geq 1,
\]
that extend the classical Ocneanu trace and give rise to HOMFLYPT-type invariants for singular links. These invariants satisfy a
desingularization relation of the form
\[
L_\times = X L_+ + Y L_0,
\]
where \(L_\times\) denotes a singular crossing, \(L_+\) a positive crossing, and \(L_0\) the smoothing.
\end{theorem}

\begin{figure}[ht]
\centering
\includegraphics[scale=1.1]{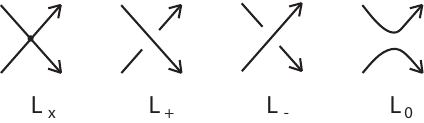}
\caption{From left to right: a singular crossing \(L_\times\), a positive crossing \(L_+\), a negative crossing \(L_-\), and the smoothing \(L_0\).}
\label{skeins}
\end{figure}

\begin{remark}
Although the singular braid monoid \(SM_n\) is algebraically isomorphic to the pseudo braid monoid \(PM_n\) via the correspondence \(\tau_i \leftrightarrow p_i\), the two theories differ at the topological level, due to the PR1 move. As a consequence, standard trace-based constructions do not directly yield invariants of pseudo links and require an additional normalization.

The construction in the present paper builds on this framework but departs from the singular case in a fundamental way. Instead of using desingularization relations, we resolve each pre-crossing into a linear combination of the two possible classical crossings. This leads to a pseudo skein relation of the form
\[
L_p = X L_+ + Y L_-,
\]
together with a normalization that ensures invariance under pseudo Markov moves.
\end{remark}


\section{The pseudo Hecke algebra and the resolution homomorphism}\label{pseudoheck}

In this section we recall the pseudo Hecke algebra of type \(A\), introduced in \cite{DiamantisPseudoST}, and construct a resolution homomorphism from this algebra to the classical Hecke algebra. This map provides the algebraic mechanism by which a pre-crossing is resolved into a linear combination of the two possible classical crossings.

\subsection{The pseudo Hecke algebra of type A}

Let \(PM_n\) denote the pseudo braid monoid on \(n\) strands, generated by classical braid generators \(\sigma_i^{\pm1}\) and pseudo generators \(p_i\) (recall Definition~\ref{def:pbrmon}), and let \(R\) be a commutative ring containing an invertible parameter \(q\). For example, one may take \(R=\mathbb{C}(q)\), or a suitable Laurent polynomial ring containing \(q^{\pm1}\).

\begin{definition}
The \emph{pseudo Hecke algebra of type \(A\)}, denoted \(PH_n(q)\), is the quotient algebra
\[
PH_n(q)
=
R[PM_n]\Big/
\left\langle
\sigma_i^2-(q-1)\sigma_i-q
\;\middle|\;
1\leq i<n
\right\rangle .
\]
We denote by \(g_i\) the image of the classical braid generator
\(\sigma_i\) in this quotient.
\end{definition}

Equivalently, \(PH_n(q)\) is generated by
\[
g_1,\ldots,g_{n-1},
\qquad
p_1,\ldots,p_{n-1},
\]
where the generators \(g_i\) satisfy the usual Hecke relations
\[
g_i^2=(q-1)g_i+q,
\qquad 1\leq i<n,
\]
together with the braid relations inherited from \(B_n\), while the pseudo generators \(p_i\) satisfy the pseudo braid relations inherited from \(PM_n\).

In particular, the pseudo generators do not satisfy any quadratic Hecke relation. This distinction is important: the classical generators become invertible in \(PH_n(q)\), since
\[
g_i^{-1}=q^{-1}g_i+(q^{-1}-1),
\]
whereas the pseudo generators \(p_i\) are not treated as invertible braid generators. This asymmetry reflects the fact that pre-crossings are not resolved at the level of the algebra.

As noted above, the pseudo braid monoid \(PM_n\) is isomorphic to the singular braid monoid \(SM_n\) under the correspondence
\[
\sigma_i^{\pm1}\longleftrightarrow \sigma_i^{\pm1},
\qquad
p_i\longleftrightarrow \tau_i,
\]
where \(\tau_i\) denotes the singular braid generator. Consequently, the pseudo Hecke algebra \(PH_n(q)\) may be viewed as an analogue of the singular Hecke algebra \(SH_n(q)\) studied in \cite{ParisRabenda2008}. However, this correspondence is purely algebraic and does not extend to an equivalence of the underlying topological theories.

The natural inclusion
\[
PM_n\hookrightarrow PM_{n+1}
\]
induces an algebra homomorphism
\[
PH_n(q)\hookrightarrow PH_{n+1}(q).
\]
Thus the pseudo Hecke algebras form a tower
\[
PH_1(q)\subset PH_2(q)\subset PH_3(q)\subset \cdots ,
\]
analogous to the classical tower of Hecke algebras used in the construction of the HOMFLYPT polynomial \cite{Jones1987}.

Since the defining Hecke relations involve only the classical generators \(g_i\), the number of pseudo generators in a word is preserved by the quotient relations. This induces a natural grading by the number of pseudo crossings.

\begin{definition}\label{def:grading}
For \(d\geq 0\), let \(P_dH_n(q)\) denote the \(R\)-submodule of \(PH_n(q)\) spanned by elements represented by pseudo braids containing exactly \(d\) pseudo generators.
\end{definition}

Then
\[
PH_n(q)
=
\bigoplus_{d\geq 0} P_dH_n(q).
\]
The integer \(d\) records the pseudo degree, or equivalently the number of pre-crossings. Similar gradings occur in the singular Hecke algebra setting
\cite{ParisRabenda2008}.

\begin{remark}
This grading is essential for the invariant constructed below. In particular, pseudo-stabilization changes the pseudo degree, and the normalization of the invariant must compensate for this variation.
\end{remark}


\subsection{The resolution homomorphism}\label{subsection:reshom}

A central idea of this paper is to interpret each pseudo crossing as a linear combination of the two possible classical crossings. Algebraically, this corresponds to resolving a pseudo generator \(p_i\) as a linear combination of the Hecke generator \(g_i\) and its inverse \(g_i^{-1}\).

\medskip 

Let \(H_n(q)\) denote the classical Hecke algebra of type \(A\), and define
\[
h_i := Xg_i + Yg_i^{-1},
\]
where \(X,Y\in R\). 

We define a map on generators by
\[
\rho_{X,Y}(g_i)=g_i,
\qquad
\rho_{X,Y}(p_i)=h_i.
\]

\begin{remark}
The above construction is reminiscent of diagrammatic approaches such as tangle insertion (cf.\ \cite{HenrichKauffman2017}), in which a crossing is replaced by a linear combination of simpler local configurations. In the present setting, this idea is realized at the algebraic level via the Hecke algebra framework.
\end{remark}

\begin{proposition}
The above assignment extends to a well-defined algebra homomorphism
\[
\rho_{X,Y}:PH_n(q)\longrightarrow H_n(q).
\]
\end{proposition}

\begin{proof}
The verification reduces to checking that the defining relations of \(PM_n\) are preserved under \(\rho_{X,Y}\), using the braid relations in \(H_n(q)\) and the fact that each \(h_i\) is a linear combination of \(g_i\) and \(g_i^{-1}\).

\medskip 

The classical braid relations among the \(g_i\)'s are preserved because \(\rho_{X,Y}\) acts as the identity on the classical generators. The Hecke quadratic relations are also preserved by definition of \(H_n(q)\).

\medskip 

It remains to check the relations involving pseudo generators. Write
\[
h_i=Xg_i+Yg_i^{-1}.
\]

First, suppose \(|i-j|\geq 2\). Then \(g_i\) commutes with \(g_j\), and hence \(g_i^{\pm1}\) commutes with \(g_j^{\pm1}\). Therefore
\[
h_i h_j=h_j h_i,
\]

\noindent and thus, the relation \(p_i p_j = p_j p_i\) is preserved under \(\rho_{X,Y}\).

\medskip 

Similarly, if \(|i-j|\geq 2\), then \(g_j^{\pm1}\) commutes with
\(g_i^{\pm1}\). Hence
\[
h_i g_j^{\pm1}=g_j^{\pm1}h_i,
\]
\noindent so the relation \(p_i\sigma_j^{\pm1}=\sigma_j^{\pm1}p_i\) is preserved under \(\rho_{X,Y}\).

\medskip 

Next, since \(h_i\) is a linear combination of \(g_i\) and \(g_i^{-1}\), it
commutes with \(g_i\) and with \(g_i^{-1}\). Therefore
\[
h_i g_i^{\pm1}=g_i^{\pm1}h_i,
\]
\noindent and the relation \(p_i\sigma_i^{\pm1}=\sigma_i^{\pm1}p_i\) is preserved.

\medskip 

It remains to check the two adjacent mixed relations. Consider first
\[
\sigma_i\sigma_{i+1}p_i=p_{i+1}\sigma_i\sigma_{i+1}.
\]
Under \(\rho_{X,Y}\), this becomes
\[
g_i g_{i+1}h_i=h_{i+1}g_i g_{i+1}.
\]
Expanding both sides, we need
\[
g_i g_{i+1}(Xg_i+Yg_i^{-1})
=
(Xg_{i+1}+Yg_{i+1}^{-1})g_i g_{i+1}.
\]
The \(X\) and the \(Y\)-terms agree because of the braid relation \(g_i g_{i+1}g_i=g_{i+1}g_i g_{i+1}.\) Thus the first adjacent mixed relation is preserved.

\medskip 

The second adjacent mixed relation is verified in the same way, using the same braid relation and multiplying by inverses as above.

\medskip 

Therefore all defining relations of \(PM_n\), and hence all defining relations of \(PH_n(q)\), are respected by \(\rho_{X,Y}\). Consequently, the assignment extends to a well-defined algebra homomorphism.
\end{proof}

\begin{remark}
The map \(\rho_{X,Y}\) should be understood as an algebraic resolution of pre-crossings. It differs from the desingularization maps in the theory of singular links, where a singular crossing is resolved into a linear combination of a classical crossing and a smoothing. In the pseudo setting, a pre-crossing represents unresolved over/under information, so the natural resolution is into the two possible classical crossings.
\end{remark}

The homomorphism \(\rho_{X,Y}\) provides the key mechanism for transferring the Ocneanu trace from the classical Hecke algebra to the pseudo Hecke algebra in the next section, leading to the construction of a HOMFLYPT-type invariant for pseudo links.


\section{Trace construction and a HOMFLYPT-type invariant}\label{trace}

In this section we construct an invariant of oriented pseudo links by combining the resolution homomorphism with the Ocneanu trace on the classical Hecke algebra. The key point is that, due to the presence of the pseudo-stabilization move, an additional normalization is required to obtain invariance under pseudo Markov moves.

\subsection{The induced trace}

Let \(\rho_{X,Y}:PH_n(q)\to H_n(q)\) be the resolution homomorphism from \S~\ref{subsection:reshom}, and let
\[
\operatorname{tr}_n:H_n(q)\to R
\]
be the Ocneanu trace on the classical Hecke algebra. We recall that this trace is uniquely characterized by the normalization
\[
\operatorname{tr}_1(1)=1,
\]
together with the properties
\[
\operatorname{tr}_{n+1}(a)=\operatorname{tr}_n(a),
\qquad
\operatorname{tr}_{n+1}(a g_n)=z\,\operatorname{tr}_n(a),
\]
for all \(a\in H_n(q)\).

\medskip 

We define a linear map \(T_n:PH_n(q)\to R\) by
\[
T_n(a)=\operatorname{tr}_n\big(\rho_{X,Y}(a)\big).
\]
This allows us to transfer the Ocneanu trace from the classical Hecke algebra to the pseudo Hecke algebra.

\begin{proposition}
The family \(\{T_n\}_{n\geq 1}\) satisfies
\[
T_n(ab)=T_n(ba),
\]
\[
T_{n+1}(a)=T_n(a),
\]
\[
T_{n+1}(a g_n)=z\,T_n(a),
\]
and
\[
T_{n+1}(a p_n)=(Xz+Yz_-)\,T_n(a),
\]
for all \(a\in PH_n(q)\), where
\[
z_-:=q^{-1}z+q^{-1}-1.
\]
\end{proposition}

\begin{proof}
The trace property follows from the trace property of the Ocneanu trace:
\[
T_n(ab)
=
\operatorname{tr}_n\big(\rho_{X,Y}(ab)\big)
=
\operatorname{tr}_n\big(\rho_{X,Y}(a)\rho_{X,Y}(b)\big).
\]
Since \(\operatorname{tr}_n\) is a trace,
\[
\operatorname{tr}_n\big(\rho_{X,Y}(a)\rho_{X,Y}(b)\big)
=
\operatorname{tr}_n\big(\rho_{X,Y}(b)\rho_{X,Y}(a)\big),
\]
and hence
\[
T_n(ab)=T_n(ba).
\]

\noindent Next, for \(a\in PH_n(q)\), viewed inside \(PH_{n+1}(q)\), we have
\[
T_{n+1}(a)
=
\operatorname{tr}_{n+1}\big(\rho_{X,Y}(a)\big)
=
\operatorname{tr}_n\big(\rho_{X,Y}(a)\big)
=
T_n(a).
\]

\noindent For the classical stabilization generator, we compute
\[
T_{n+1}(a g_n)
=
\operatorname{tr}_{n+1}\big(\rho_{X,Y}(a g_n)\big).
\]
Since \(\rho_{X,Y}(g_n)=g_n\), this becomes
\[
T_{n+1}(a g_n)
=
\operatorname{tr}_{n+1}\big(\rho_{X,Y}(a)g_n\big).
\]
Using the Markov property of the Ocneanu trace,
\[
\operatorname{tr}_{n+1}\big(\rho_{X,Y}(a)g_n\big)
=
z\,\operatorname{tr}_n\big(\rho_{X,Y}(a)\big).
\]
Thus
\[
T_{n+1}(a g_n)=z\,T_n(a).
\]

Finally, for the pseudo-stabilization generator, we compute
\[
T_{n+1}(a p_n)
=
\operatorname{tr}_{n+1}\big(\rho_{X,Y}(a p_n)\big).
\]
Since
\[
\rho_{X,Y}(p_n)=Xg_n+Yg_n^{-1},
\]
we get
\[
T_{n+1}(a p_n)
=
\operatorname{tr}_{n+1}\big(\rho_{X,Y}(a)(Xg_n+Yg_n^{-1})\big).
\]
By linearity,
\[
T_{n+1}(a p_n)
=
X\,\operatorname{tr}_{n+1}\big(\rho_{X,Y}(a)g_n\big)
+
Y\,\operatorname{tr}_{n+1}\big(\rho_{X,Y}(a)g_n^{-1}\big).
\]

\noindent Using the Hecke relation, \(g_n^{-1}=q^{-1}g_n+(q^{-1}-1),\) we obtain
\[
\operatorname{tr}_{n+1}\big(\rho_{X,Y}(a)g_n^{-1}\big)
=
q^{-1}\operatorname{tr}_{n+1}\big(\rho_{X,Y}(a)g_n\big)
+
(q^{-1}-1)\operatorname{tr}_{n+1}\big(\rho_{X,Y}(a)\big).
\]
Therefore,
\[
\operatorname{tr}_{n+1}\big(\rho_{X,Y}(a)g_n^{-1}\big)
=
q^{-1}z\,\operatorname{tr}_n\big(\rho_{X,Y}(a)\big)
+
(q^{-1}-1)\operatorname{tr}_n\big(\rho_{X,Y}(a)\big).
\]
Hence
\[
\operatorname{tr}_{n+1}\big(\rho_{X,Y}(a)g_n^{-1}\big)
=
\left(q^{-1}z+q^{-1}-1\right)
\operatorname{tr}_n\big(\rho_{X,Y}(a)\big).
\]
Thus, with
\[
z_-:=q^{-1}z+q^{-1}-1,
\]
we have
\[
\operatorname{tr}_{n+1}\big(\rho_{X,Y}(a)g_n^{-1}\big)
=
z_-\,T_n(a).
\]

Combining the two terms gives
\[
T_{n+1}(a p_n)
=
Xz\,T_n(a)+Yz_-\,T_n(a),
\]
and therefore
\[
T_{n+1}(a p_n)
=
(Xz+Yz_-)\,T_n(a).
\]
\end{proof}

\begin{remark}
The coefficient \( Xz+Yz_-\) represents the trace-theoretic contribution of resolving a terminal pre-crossing into a combination of positive and negative stabilizations. This is precisely the factor that must be compensated by the pseudo normalization in the definition of the invariant.
\end{remark}


\subsection{Definition of the invariant}

Let \(\alpha\in PM_n\). For a word representative of \(\alpha\), let \(e(\alpha)\) denote the exponent sum of the classical generators \(\sigma_i\), and let \(d(\alpha)\) denote the number of pseudo generators \(p_i\) appearing in the word. The following proposition shows that these quantities are independent of the chosen representative.

\begin{proposition}
Let \(\alpha\in PM_n\). The number \(d(\alpha)\) of pseudo generators appearing in a word representative of \(\alpha\) is well-defined. Moreover, the classical exponent sum \(e(\alpha)\) is also well-defined.
\end{proposition}

\begin{proof}
We check that the defining relations of \(PM_n\) preserve both quantities.

The classical braid relations involve only the generators
\(\sigma_i^{\pm1}\), and therefore do not affect the number of pseudo
generators. They also preserve the total exponent sum of the classical
generators.

The far-commutativity relations
\[
p_i p_j=p_j p_i,
\qquad |i-j|\geq 2,
\]
preserve the number of pseudo generators, since both sides contain exactly
two pseudo generators.

Similarly, the relations
\[
p_i\sigma_j^{\pm1}=\sigma_j^{\pm1}p_i,
\qquad |i-j|\geq 2,
\]
and
\[
p_i\sigma_i^{\pm1}=\sigma_i^{\pm1}p_i
\]
preserve the number of pseudo generators and the classical exponent sum.

Finally, the adjacent mixed relations
\[
\sigma_i\sigma_{i+1}p_i=p_{i+1}\sigma_i\sigma_{i+1}
\]
and
\[
\sigma_{i+1}\sigma_i p_{i+1}=p_i\sigma_{i+1}\sigma_i
\]
also preserve both quantities: each side contains exactly one pseudo
generator and two positive classical generators.

Thus all defining relations preserve \(d(\alpha)\) and \(e(\alpha)\).
Hence both are well-defined functions on \(PM_n\).
\end{proof}

\begin{remark}
The integer \(d(\alpha)\) coincides with the degree of \(\alpha\) with respect to the natural grading introduced Definition~\ref{def:grading}.
\end{remark}

\begin{definition}
Let \(\alpha \in PM_n\). We define a function
\[
\mathcal P(\widehat{\alpha})
=
A^{n-1} B^{e(\alpha)} C^{d(\alpha)} T_n(\alpha),
\]
where \(A,B,C \in R\) are parameters.
\end{definition}

We now determine the conditions on \(A,B,C\) that ensure that \(\mathcal P\) is invariant under the pseudo Markov moves.

\paragraph{Positive stabilization.}
Let \(\alpha \in PM_n\). Then
\[
\mathcal P(\widehat{\alpha \sigma_n})
=
A^{n} B^{e(\alpha)+1} C^{d(\alpha)} T_{n+1}(\alpha \sigma_n).
\]
Using \(T_{n+1}(\alpha \sigma_n)=z\,T_n(\alpha),\) we obtain
\[
\mathcal P(\widehat{\alpha \sigma_n})
=
A^{n} B^{e(\alpha)+1} C^{d(\alpha)} z\,T_n(\alpha).
\]
Comparing with \(\mathcal P(\widehat{\alpha})
=
A^{n-1} B^{e(\alpha)} C^{d(\alpha)} T_n(\alpha),
\) we get \(\mathcal P(\widehat{\alpha \sigma_n})
=
ABz\,\mathcal P(\widehat{\alpha}).
\) Thus we require
\begin{equation}\label{eqpos}
  ABz=1.  
\end{equation}

\paragraph{Negative stabilization.}
Similarly, 
\[\mathcal P(\widehat{\alpha \sigma_n^{-1}})
=
A^{n} B^{e(\alpha)-1} C^{d(\alpha)} T_{n+1}(\alpha \sigma_n^{-1}).
\]
Using \(T_{n+1}(\alpha \sigma_n^{-1})=z_-\,T_n(\alpha),
\) we obtain
\[
\mathcal P(\widehat{\alpha \sigma_n^{-1}})
=
A^{n} B^{e(\alpha)-1} C^{d(\alpha)} z_-\,T_n(\alpha),
\]
and hence
\(
\mathcal P(\widehat{\alpha \sigma_n^{-1}})
=
AB^{-1}z_-\,\mathcal P(\widehat{\alpha}).
\) Thus we require
\begin{equation}\label{eqpneg}
AB^{-1}z_-=1.
\end{equation}

\paragraph{Pseudo-stabilization.}
Finally,
\[
\mathcal P(\widehat{\alpha p_n})
=
A^{n} B^{e(\alpha)} C^{d(\alpha)+1} T_{n+1}(\alpha p_n).
\]
Using \(T_{n+1}(\alpha p_n)=(Xz+Yz_-)\,T_n(\alpha),
\) we obtain 
\[
\mathcal P(\widehat{\alpha p_n})
=
A^{n} B^{e(\alpha)} C^{d(\alpha)+1} (Xz+Yz_-)\,T_n(\alpha).
\] 
Comparing with \(\mathcal P(\widehat{\alpha})\), we get \(\mathcal P(\widehat{\alpha p_n})
=
AC(Xz+Yz_-)\,\mathcal P(\widehat{\alpha}),\) and therefore we require
\begin{equation}\label{eqppseudo}
AC(Xz+Yz_-)=1.
\end{equation}

Solving equations~(\ref{eqpos}), (\ref{eqpneg}) and (\ref{eqppseudo}) yields
\begin{equation}\label{eqcoef}
B^2=\frac{z_-}{z},
\qquad
A=\frac{1}{Bz},
\qquad
C=\frac{1}{A(Xz+Yz_-)}.
\end{equation}

Throughout the construction, we work over a coefficient ring \(R\) containing
the parameters \(q^{\pm1},z,X,Y\), and in which the elements
\[
z,\qquad z_-:=q^{-1}z+q^{-1}-1,\qquad Xz+Yz_-
\]
are invertible. We also assume that \(R\) contains an element \(B\) satisfying \( B^2=\frac{z_-}{z}. \) We then set
\[
A=\frac{1}{Bz},
\qquad
C=\frac{1}{A(Xz+Yz_-)}.
\]


\begin{theorem}
Let \(\alpha \in PM_n\), and define
\[
\mathcal P(\widehat{\alpha})
=
A^{n-1} B^{e(\alpha)} C^{d(\alpha)} T_n(\alpha),
\]
where \(T_n=\operatorname{tr}_n\circ \rho_{X,Y}\) and the constants
\(A,B,C\) satisfy Eq.~\eqref{eqcoef}. Then \(\mathcal P\) is an invariant of oriented pseudo links.
\end{theorem}

\begin{proof}
By the pseudo Alexander theorem, every pseudo link arises as the closure of a pseudo braid. By the pseudo Markov theorem, two pseudo braids have isotopic closures if and only if they are related by a finite sequence of conjugation, commuting, classical and pseudo-stabilization moves.

The trace property of \(T_n\) implies invariance under conjugation and commuting. The normalization conditions derived above ensure invariance under positive and negative stabilization, as well as pseudo-stabilization.

Therefore, \(\mathcal P\) is invariant under all pseudo Markov moves, and hence defines an invariant of oriented pseudo links.
\end{proof}

\begin{remark}
The invariant \(\mathcal P\) extends the classical HOMFLYPT construction. The presence of pseudo crossings introduces an additional degree, encoded by \(d(\alpha)\), which necessitates the normalization factor \(C^{d(\alpha)}\) and which reflects the role of the pseudo Reidemeister 1 move.
\end{remark}


\section{Skein relations and examples}\label{skein}

In this section we derive the skein relations satisfied by the invariant \(\mathcal P\) and present explicit computations illustrating its behavior. In particular, we show that \(\mathcal P\) satisfies a natural pseudo skein relation corresponding to the algebraic resolution of pre-crossings introduced in Section~\ref{pseudoheck}.

\subsection{Classical skein relation}

We recall that the Hecke relation \(g_i^2 = (q-1)g_i + q\) implies that \(g_i\) is invertible, with \(g_i^{-1} = q^{-1} g_i + (q^{-1}-1)\). This identity is the algebraic origin of the classical HOMFLYPT skein relation.

\medskip

Let \(L_+, L_-, L_0\) be oriented link diagrams which coincide outside a small disk and differ inside the disk by a positive crossing, a negative crossing, and the oriented smoothing, respectively (recall Figure~\ref{skeins}).

\begin{proposition}
The invariant \(\mathcal P\) satisfies a HOMFLYPT-type skein relation
\[
a\,\mathcal P(L_+) - a^{-1}\,\mathcal P(L_-)
=
b\,\mathcal P(L_0),
\]
where \(a,b\) depend on the parameters \(q,z\) and the normalization constants \(A,B\).
\end{proposition}

\begin{remark}
The coefficients \(a\) and \(b\) can be computed explicitly from the relations \(T_{n+1}(a g_n)=z\,T_n(a)\) and \(g_n^{-1}=q^{-1}g_n+(q^{-1}-1)\), together with the normalization of \(\mathcal P\). Since this coincides formally with the classical HOMFLYPT skein relation, we omit the computation and refer to \cite{Jones1987, HOMFLY1985}.
\end{remark}


\subsection{Pseudo skein relation}

We now derive the skein relation associated with a pre-crossing.

\medskip 

Let \(L_p, L_+, L_-\) be three oriented link diagrams which are identical outside a small disk, and inside the disk differ by a pre-crossing, a positive crossing, and a negative crossing, respectively.

At the braid level, these diagrams correspond to elements obtained by replacing \(p_i\) with \(g_i\) or \(g_i^{-1}\). Using the definition of the resolution homomorphism \(\rho_{X,Y}\), we have
\[
T_n(\alpha p_i \beta)
=
X\,T_n(\alpha g_i \beta)
+
Y\,T_n(\alpha g_i^{-1} \beta).
\]

We now express this relation in terms of the normalized invariant
\(\mathcal P\). Let \(\alpha p_i \beta \in PM_n\) be a braid representative corresponding to these diagrams.

Let \(e := e(\alpha\beta)\) and \(d := d(\alpha\beta)\). Then
\[
\mathcal P(L_p)
=
A^{n-1} B^{e} C^{d+1}
\Big(
X\,T_n(\alpha g_i \beta)
+
Y\,T_n(\alpha g_i^{-1} \beta)
\Big).
\]

On the other hand,
\[
\mathcal P(L_+)
=
A^{n-1} B^{e+1} C^{d}
\,T_n(\alpha g_i \beta),
\]
and
\[
\mathcal P(L_-)
=
A^{n-1} B^{e-1} C^{d}
\,T_n(\alpha g_i^{-1} \beta).
\]

Solving for the trace terms, we obtain
\[
A^{n-1} B^{e} C^{d+1} T_n(\alpha g_i \beta)
=
C B^{-1} \,\mathcal P(L_+),
\]
and
\[
A^{n-1} B^{e} C^{d+1} T_n(\alpha g_i^{-1} \beta)
=
C B \,\mathcal P(L_-).
\]

It is convenient to introduce the coefficients
\[
\lambda_+ := X C B^{-1},
\qquad
\lambda_- := Y C B,
\]
and this leads to the following result:

\begin{proposition}[The pseudo skein relation]
The invariant \(\mathcal P\) satisfies
\[
\mathcal P(L_p)
=
\lambda_+ \,\mathcal P(L_+)
+
\lambda_- \,\mathcal P(L_-).
\]
\end{proposition}

\begin{remark}
Using the normalization conditions, the coefficients \(CB^{-1}\) and \(CB\) can be expressed explicitly in terms of the parameters \(q,z,X,Y\). Thus, the pseudo skein relation is entirely determined by the algebraic structure of the construction.
\end{remark}


\subsection{Examples}

We now compute the invariant \(\mathcal P\) for several simple examples. These computations illustrate the role of the normalization factors and show how pseudo crossings contribute through the resolution map.

\paragraph{The unknot.}
Let \(U\) denote the unknot, represented by the closure of the trivial braid on one strand. Since \(n=1\), \(e(1)=0\), and \(d(1)=0\), we have
\[
\mathcal P(U)
=
A^{0}B^{0}C^{0}T_1(1)
=
T_1(1).
\]
With the normalization \(T_1(1)=1\), this gives
\[
\mathcal P(U)=1.
\]

\paragraph{A single pseudo crossing.}
Consider the closure of the braid \(p_1\in PM_2\). This represents an unknot with a single pre-crossing, corresponding to the pseudo Reidemeister 1 move.
Here
\[
n=2,\qquad e(p_1)=0,\qquad d(p_1)=1.
\]
Therefore
\[
\mathcal P(\widehat{p_1})
=
A C\,T_2(p_1).
\]
Using the resolution map,
\[
\rho_{X,Y}(p_1)=Xg_1+Yg_1^{-1},
\]
we obtain
\[
T_2(p_1)
=
X\,\operatorname{tr}_2(g_1)
+
Y\,\operatorname{tr}_2(g_1^{-1}).
\]
By definition,
\[
\operatorname{tr}_2(g_1)=z
\]
and
\[
\operatorname{tr}_2(g_1^{-1})=z_-.
\]
Hence
\[
T_2(p_1)=Xz+Yz_-.
\]
Thus
\[
\mathcal P(\widehat{p_1})
=
AC(Xz+Yz_-).
\]
Using the normalization condition
\[
AC(Xz+Yz_-)=1,
\]
we obtain
\[
\mathcal P(\widehat{p_1})=1.
\]
This agrees with the fact that a single pre-crossing kink is removed by the pseudo Reidemeister 1 move.

\paragraph{A single classical crossing.}
Consider the braid \(g_1\in PM_2\). Here
\[
n=2,\qquad e(g_1)=1,\qquad d(g_1)=0.
\]
Therefore
\[
\mathcal P(\widehat{g_1})
=
AB\,T_2(g_1).
\]
Since
\[
T_2(g_1)=\operatorname{tr}_2(g_1)=z,
\]
we get
\[
\mathcal P(\widehat{g_1})
=
ABz.
\]
By the normalization condition
\[
ABz=1,
\]
we obtain
\[
\mathcal P(\widehat{g_1})=1.
\]

\paragraph{A mixed example.}
Consider the braid \(g_1p_1\in PM_2\). In this case
\[
n=2,\qquad e(g_1p_1)=1,\qquad d(g_1p_1)=1.
\]
Thus
\[
\mathcal P(\widehat{g_1p_1})
=
ABC\,T_2(g_1p_1).
\]
We compute
\[
T_2(g_1p_1)
=
\operatorname{tr}_2\big(\rho_{X,Y}(g_1p_1)\big).
\]
Since
\[
\rho_{X,Y}(g_1)=g_1
\qquad\text{and}\qquad
\rho_{X,Y}(p_1)=Xg_1+Yg_1^{-1},
\]
we have
\[
\rho_{X,Y}(g_1p_1)
=
g_1(Xg_1+Yg_1^{-1}).
\]
Therefore
\[
T_2(g_1p_1)
=
\operatorname{tr}_2\big(g_1(Xg_1+Yg_1^{-1})\big)
\]
and hence
\[
T_2(g_1p_1)
=
X\,\operatorname{tr}_2(g_1^2)
+
Y\,\operatorname{tr}_2(1).
\]
Using the Hecke relation
\[
g_1^2=(q-1)g_1+q,
\]
we obtain
\[
\operatorname{tr}_2(g_1^2)
=
(q-1)\operatorname{tr}_2(g_1)
+
q\,\operatorname{tr}_2(1).
\]
Since
\[
\operatorname{tr}_2(g_1)=z
\qquad\text{and}\qquad
\operatorname{tr}_2(1)=1,
\]
it follows that
\[
\operatorname{tr}_2(g_1^2)
=
(q-1)z+q.
\]
Thus
\[
T_2(g_1p_1)
=
X\big((q-1)z+q\big)+Y,
\]
and therefore
\[
\mathcal P(\widehat{g_1p_1})
=
ABC\Big(X\big((q-1)z+q\big)+Y\Big).
\]

\paragraph{A family of mixed braids.}
We now consider a family illustrating the non-triviality of the invariant.

Let
\[
\alpha_k = g_1^k p_1 \in PM_2,
\qquad k \in \mathbb{Z}.
\]
Then
\[
\mathcal P(\widehat{\alpha_k})
=
A B^k C\, T_2(g_1^k p_1).
\]

Using the resolution map, we compute
\[
T_2(g_1^k p_1)
=
\operatorname{tr}_2\big(g_1^k (X g_1 + Y g_1^{-1})\big),
\]
and hence
\[
T_2(g_1^k p_1)
=
X\,\operatorname{tr}_2(g_1^{k+1})
+
Y\,\operatorname{tr}_2(g_1^{k-1}).
\]

Therefore,
\[
\mathcal P(\widehat{\alpha_k})
=
A B^k C
\Big(
X\,\operatorname{tr}_2(g_1^{k+1})
+
Y\,\operatorname{tr}_2(g_1^{k-1})
\Big).
\]

The family \(\alpha_k = g_1^k p_1\) shows that the invariant depends non-trivially on the interaction between classical and pseudo crossings.


\section{A state-sum interpretation of the invariant}\label{statesum}

In this section we show that the invariant \(\mathcal P\) admits a
state-sum description obtained by resolving each pseudo crossing into classical crossings. This provides a combinatorial interpretation of the construction and makes explicit the role of the resolution map.

\subsection{States and classical resolutions}

Let \(\alpha \in PM_n\) be a pseudo braid with \(d(\alpha)=d\) pseudo generators.

\begin{definition}
A \emph{state} \(s\) of \(\alpha\) is a choice, for each pseudo generator \(p_i\) appearing in \(\alpha\), of either the positive generator \(g_i\) or the negative generator \(g_i^{-1}\).

Given a state \(s\), we denote by \(\alpha_s \in B_n\) the classical braid obtained by replacing each occurrence of \(p_i\) in \(\alpha\) according to this choice.
\end{definition}

For a state \(s\), let \(r_+(s)\) denote the number of pseudo generators replaced by \(g_i\), and let \(r_-(s)\) denote the number replaced by \(g_i^{-1}\). Then
\[
r_+(s) + r_-(s) = d.
\]
Each state \(s\) thus determines a complete classical resolution of the pseudo braid \(\alpha\).


\subsection{Expansion of the resolution homomorphism}

We first express the resolution homomorphism in terms of states.

\begin{lemma}\label{lem:resexp}
Let \(\alpha \in PM_n\) be a pseudo braid with \(d\) pseudo generators. Then
\[
\rho_{X,Y}(\alpha)
=
\sum_s X^{r_+(s)} Y^{r_-(s)} \,\alpha_s,
\]
where the sum is taken over all states \(s\) of \(\alpha\), and \(\alpha_s\) is viewed in \(H_n(q)\) via the canonical projection \(B_n \to H_n(q)\).
\end{lemma}

\begin{proof}
By definition,
\[
\rho_{X,Y}(p_i) = X g_i + Y g_i^{-1}.
\]
Applying \(\rho_{X,Y}\) to \(\alpha\) replaces each pseudo generator by a linear combination of \(g_i\) and \(g_i^{-1}\). Expanding the resulting product yields a sum over all choices of replacing each \(p_i\) by either \(g_i\) or \(g_i^{-1}\).

Each such choice corresponds to a state \(s\), and contributes a term of the form
\[
X^{r_+(s)} Y^{r_-(s)} \,\alpha_s.
\]
Summing over all states gives the result.
\end{proof}


\subsection{State-sum formulation and relation to the HOMFLYPT invariant}

Let \(\mathcal H\) denote the normalized classical HOMFLYPT-type invariant, defined by
\[
\mathcal H(\widehat{\beta})
=
A^{n-1} B^{e(\beta)} \operatorname{tr}_n\big(\pi(\beta)\big),
\qquad \beta \in B_n,
\]
where \(\pi: B_n \to H_n(q)\) is the canonical projection.

\begin{lemma}\label{lem:resolution-expansion}
Let \(s\) be a state of \(\alpha \in PM_n\). Then
\[
e(\alpha_s) = e(\alpha) + r_+(s) - r_-(s).
\]
\end{lemma}

\begin{proof}
Each replacement \(p_i \mapsto g_i\) contributes \(+1\) to the exponent sum, while each replacement \(p_i \mapsto g_i^{-1}\) contributes \(-1\). Summing over all pseudo generators yields the result.
\end{proof}


We now obtain the desired state-sum expression.

\begin{theorem}[State-sum formula]
Let \(\alpha \in PM_n\) be a pseudo braid with \(d(\alpha)=d\) pseudo generators. Then
\[
\mathcal P(\widehat{\alpha})
=
C^d
\sum_s
\left(X B^{-1}\right)^{r_+(s)}
\left(Y B\right)^{r_-(s)}
\mathcal H(\widehat{\alpha_s}),
\]
where the sum runs over all states \(s\) of \(\alpha\).
\end{theorem}

\begin{proof}
By definition,
\[
\mathcal P(\widehat{\alpha})
=
A^{n-1} B^{e(\alpha)} C^d
\, T_n(\alpha),
\]
where
\[
T_n(\alpha)
=
\operatorname{tr}_n\big(\rho_{X,Y}(\alpha)\big).
\]

Using the expansion of the resolution homomorphism (Lemma~\ref{lem:resexp}), we obtain
\[
T_n(\alpha)
=
\sum_s X^{r_+(s)} Y^{r_-(s)} \operatorname{tr}_n\big(\pi(\alpha_s)\big).
\]

Thus
\[
\mathcal P(\widehat{\alpha})
=
A^{n-1} B^{e(\alpha)} C^d
\sum_s
X^{r_+(s)} Y^{r_-(s)}
\operatorname{tr}_n\big(\pi(\alpha_s)\big).
\]

Now, by the definition of \(\mathcal H\),
\[
\operatorname{tr}_n\big(\pi(\alpha_s)\big)
=
A^{-(n-1)} B^{-e(\alpha_s)}
\mathcal H(\widehat{\alpha_s}).
\]

Substituting this into the previous expression yields
\[
\mathcal P(\widehat{\alpha})
=
C^d
\sum_s
X^{r_+(s)} Y^{r_-(s)}
B^{e(\alpha)-e(\alpha_s)}
\mathcal H(\widehat{\alpha_s}).
\]

Using the identity from Lemma~\ref{lem:resolution-expansion},
\[
e(\alpha)-e(\alpha_s) = -(r_+(s)-r_-(s)),
\]
we obtain
\[
B^{e(\alpha)-e(\alpha_s)}
=
B^{-r_+(s)} B^{r_-(s)}.
\]

Combining all factors gives
\[
\mathcal P(\widehat{\alpha})
=
C^d
\sum_s
\left(X B^{-1}\right)^{r_+(s)}
\left(Y B\right)^{r_-(s)}
\mathcal H(\widehat{\alpha_s}),
\]
as required.
\end{proof}

\begin{remark}
The above formula shows that the invariant \(\mathcal P\) can be interpreted as a weighted sum of classical HOMFLYPT-type invariants over all classical resolutions of the pseudo crossings. This is reminiscent of state-sum constructions such as the Kauffman bracket \cite{Kauffman1988}, although in the present case the weights arise from the Hecke algebra structure and the Markov trace rather than from a purely diagrammatic smoothing rule.
\end{remark}


\subsection{Skein-theoretic characterization}

The state-sum formula suggests that the invariant \(\mathcal P\) is determined by its values on classical links together with the pseudo skein relation. We make this precise in the following result.

\begin{theorem}[Skein-theoretic characterization]
The invariant \(\mathcal P\) is uniquely determined by the following properties:
\begin{enumerate}
    \item On classical links, \(\mathcal P\) agrees with the normalized
    HOMFLYPT-type invariant \(\mathcal H\).
    
    \item It satisfies the pseudo skein relation
    \[
    \mathcal P(L_p)
    =
    \lambda_+ \mathcal P(L_+)
    +
    \lambda_- \mathcal P(L_-).
    \]
    
    \item It is normalized by
    \[
    \mathcal P(U)=1,
    \]
    where \(U\) denotes the unknot (this is consistent with (1)).
\end{enumerate}
\end{theorem}

\begin{proof}
Let \(F\) be any pseudo link invariant satisfying properties (1)--(3). We show that \(F(L)\) is uniquely determined for every pseudo link \(L\).

Let \(L\) be represented by a pseudo link diagram with \(d\) pre-crossings. If \(d=0\), then \(L\) is classical, and \(F(L)=\mathcal H(L)\) is determined by condition (1).

Suppose now that \(d>0\). Choose a pre-crossing of \(L\). By the pseudo skein relation,
\[
F(L)
=
\lambda_+ F(L_+)
+
\lambda_- F(L_-),
\]
where \(L_+\) and \(L_-\) are obtained by replacing the chosen pre-crossing by a positive and a negative crossing, respectively. Both \(L_+\) and \(L_-\) have one fewer pre-crossing than \(L\).

Proceeding by induction on the number of pre-crossings, the values
\(F(L_+)\) and \(F(L_-)\) are uniquely determined, and hence \(F(L)\) is uniquely determined. Repeating this process reduces \(F(L)\) to a finite linear combination of values of \(F\) on classical links, which are determined by condition (1).

Therefore, for every pseudo link \(L\), the value \(F(L)\) is uniquely determined by properties (1)--(3). It follows that any two invariants satisfying these properties must agree on all pseudo links. In particular, the invariant \(\mathcal P\) is uniquely determined by these properties.
\end{proof}

\begin{remark}
This characterization shows that the invariant \(\mathcal P\) can be defined purely in skein-theoretic terms, independently of the underlying Hecke algebra construction. In particular, it provides a natural extension of the HOMFLYPT-type invariant \cite{HOMFLY1985} to the setting of pseudo links.
\end{remark}


\section{Conclusion and further directions}\label{concl}

In this paper, we introduced a HOMFLYPT-type invariant for oriented pseudo links using a Hecke algebra framework. The construction is based on a resolution homomorphism from the pseudo Hecke algebra to the classical Hecke algebra, in which each pseudo generator is interpreted as a linear combination of a braid generator and its inverse. Composing this map with the Ocneanu trace and applying a suitable normalization yields an invariant of pseudo links that extends the classical HOMFLYPT polynomial.

A central feature of the construction is the pseudo skein relation
\[
\mathcal P(L_p)
=
\lambda_+ \mathcal P(L_+)
+
\lambda_- \mathcal P(L_-),
\]
which reflects the interpretation of pseudo crossings as unresolved over/under information. In contrast to the singular setting, where desingularization involves smoothings, the present framework resolves pseudo crossings into classical crossings, leading to a distinct skein-theoretic structure.

The invariant \(\mathcal P\) admits multiple equivalent descriptions: an algebraic formulation via Hecke algebras, a state-sum expansion over classical resolutions, and a skein-theoretic characterization. Together, these perspectives provide a unified understanding of pseudo link invariants and clarify their relationship with classical knot theory.

The construction also suggests several directions for further study. A natural direction is the extension of this framework to pseudo links in more general 3-manifolds using braid-theoretic techniques, following approaches such as \cite{Lambropoulou1999,DiamantisPseudoST}. In this context, skein-theoretic methods and braid representations provide a natural setting for further developments. Related directions include recent work on generalized braid structures, such as bonded braids \cite{DKL}, suggesting the existence of a broader algebraic framework encompassing pseudo, singular and other generalized knot theories.

From an algebraic viewpoint, it would be of independent interest to study the pseudo Hecke algebra \(PH_n(q)\) in greater depth. In particular, establishing a basis for \(PH_n(q)\), as conjectured in \cite{DiamantisPseudoST}, would be an important step toward constructing intrinsic Markov-type traces directly on \(PH_n(q)\), without passing through the classical Hecke algebra.


\begin{thebibliography}{99}

\bibitem{Alexander1923}
J. W. Alexander,
\newblock A lemma on systems of knotted curves,
\newblock \emph{Proc. Natl. Acad. Sci. USA}, 9 (1923), 93--95.

\bibitem{Baez1992}
J. C. Baez,
\newblock Link invariants of finite type and perturbation theory,
\newblock \emph{Lett. Math. Phys.}, 26 (1992), 43--51.

\bibitem{BardakovJablanWang2016}
V. G. Bardakov, S. Jablan, and H. Wang,
\newblock Monoid and group of pseudo braids,
\newblock \emph{J. Knot Theory Ramifications}, 25 (2016), 1641002.

\bibitem{Birman1993}
J. S. Birman,
\newblock New points of view in knot theory,
\newblock \emph{Bull. Amer. Math. Soc.}, 28 (1993), 253--287.

\bibitem{D}
I.~Diamantis,
\newblock Tied pseudo links \& pseudo knotoids,
\newblock \emph{Mediterr. J. Math.} \textbf{18} (2021), 201.

\bibitem{DiamantisPseudoST}
I. Diamantis,
\newblock Pseudo links and singular links in the solid torus,
\newblock \emph{Communications in Mathematics}, \textbf{31} (2023), 333--357.

\bibitem{DKL}
I.~Diamantis, L.~H.~Kauffman, and S.~Lambropoulou,
\newblock Topology and algebra of bonded knots and braids,
\newblock \emph{Mathematics} \textbf{13} (2025), 3260.

\bibitem{HOMFLY1985}
P. Freyd, D. Yetter, J. Hoste, W. B. R. Lickorish, K. Millett, and A. Ocneanu,
\newblock A new polynomial invariant of knots and links,
\newblock \emph{Bull. Amer. Math. Soc.} \textbf{12} (1985), 239--246.

\bibitem{Hanaki2010}
R. Hanaki,
\newblock Pseudo diagrams of knots, links and spatial graphs,
\newblock \emph{Osaka J. Math.}, 47 (2010), 863--883.

\bibitem{Henrich2013}
A. Henrich, R. Hoberg, S. Jablan, L. Johnson, E. Minten, and L. Radović,
\newblock The theory of pseudoknots,
\newblock \emph{J. Knot Theory Ramifications}, 22 (2013), 1350032.

\bibitem{HenrichKauffman2017}
A.~Henrich and L.~H.~Kauffman,
\newblock Tangle insertion invariants for pseudoknots, singular knots, and rigid vertex spatial graphs,
\newblock \emph{Contemp. Math.} \textbf{689} (2017), 177--189.

\bibitem{Jones1987}
V. F. R. Jones,
\newblock Hecke algebra representations of braid groups and link polynomials,
\newblock \emph{Ann. of Math.}, 126 (1987), 335--388.

\bibitem{Kauffman1988}
L. H. Kauffman,
\newblock New invariants in the theory of knots,
\newblock \emph{Amer. Math. Monthly}, 95 (1988), 195--242.

\bibitem{Lambropoulou1999}
S. Lambropoulou,
\newblock Knot theory related to generalized and cyclotomic Hecke algebras of type B,
\newblock \emph{J. Knot Theory Ramifications}, 8 (1999), 621--658.

\bibitem{Lambropoulou2007}
S. Lambropoulou,
\newblock L-moves and Markov theorems,
\newblock \emph{J. Knot Theory Ramifications}, 16 (2007), 1459--1468.

\bibitem{Markov1936}
A.~A.~Markov,
\newblock Über die freie Äquivalenz der geschlossenen Zöpfe,
\newblock \emph{Rec. Math. [Mat. Sbornik] N.S.} \textbf{1(43)} (1936), 73--78.

\bibitem{ParisRabenda2008}
L. Paris and L. Rabenda,
\newblock Singular Hecke algebras, Markov traces and HOMFLY-type invariants,
\newblock \emph{Ann. Inst. Fourier}, 58 (2008), 2413--2443.

\end{thebibliography}
\end{document}